\documentclass[11pt]{amsart}
\usepackage[parfill]{parskip}    

\usepackage{url}
\usepackage{amsmath}
\usepackage{graphicx}
\usepackage{float}
\usepackage{amssymb}
\usepackage{booktabs}
\usepackage{mathtools}
\usepackage[T1]{fontenc}
\usepackage{tikz}
\usepackage{enumerate}
\usepackage{bm}
\usepackage[hidelinks]{hyperref}

\newcommand*{\rom}[1]{\expandafter\@slowromancap\romannumeral #1@}
\newcommand{\RNum}[1]{\uppercase\expandafter{\romannumeral #1\relax}}
\newcommand{\Z}{\mathbb{Z}}  

\DeclarePairedDelimiter\floor{\lfloor}{\rfloor}
\newcommand\labs[1]{\left\lvert#1\right\rvert}

\newcommand{\Split}{\mathrm{Split}}

\newcommand{\inte}{\mathrm{int}\,}

\renewcommand{\leq}{\leqslant}
\renewcommand{\geq}{\geqslant}

\theoremstyle{plain} 
\newtheorem{thm}{Theorem}[section]
\newtheorem{lem}[thm]{Lemma}
\newtheorem{pro}[thm]{Proposition}

\newtheorem{qu}[thm]{Question}

\newtheorem{thmAlph}{Theorem}[section]

\newtheorem{corAlph}[thmAlph]{Corollary}

\theoremstyle{definition} 

\newtheorem{defn}[thm]{Definition}

\theoremstyle{remark} 
\newtheorem{rema}[thm]{Remark}
\newtheorem*{ack}{Acknowledgements}

\title{Ramsey Theory for Product Trees}
\author{Alexander Fish, Sean Skinner}
\address{School of Mathematics and Statistics, University of Sydney, Australia}
\curraddr{}
\email{alexander.fish@sydney.edu.au}
\email{sean.skinner@sydney.edu.au}
\thanks{}

\date{\today}                                           

\begin{document}

\maketitle

\begin{abstract}
We develop a Ramsey theory for leaf-generated subsets of finite products of trees. Our starting point is a Theorem of Furstenberg and Weiss which states that for every $k\geq 1$ and $\alpha>0$, if $A$ is a subset of the leaves of $T_n$, where $T_n$ is the complete binary tree of height $n$, with size $\labs{A}\geq 2^{\alpha n}$, then for $n$ sufficiently large the ancestor closed sub-tree $T_A\subset T_n$ generated by $A$ must contain a copy of $T_k$ such that (1) all vertices in the same level of $T_k$ are mapped into vertices at the same level of $T_A$, (2) if a non-leaf vertex $x\in T_k$ is mapped into a vertex $y \in T_A$, then the two children of $x$ are mapped into descendants of the two children of $y$, and (3) the levels of $T_A$ occupied by the copy of $T_k$ form an arithmetic progression.

For the product $T_n \times T_n$ of two binary trees, we show that for every $k\geq 1$ and $\alpha > 1$, for $n$ sufficiently large, every subset $A$ of the leaves of $T_n\times T_n$ of size at least $2^{\alpha n}$ has that its ancestor closure $\Gamma_A\subset T_n \times T_n$ contains similarly structured arithmetic copies of the height $k$ four-ary tree, where the four children of every non-leaf vertex in the source tree are required to map below the four distinct \textit{diagonal children} of the image vertex, where a diagonal child of a vertex $(x,y) \in T_n \times T_n$ is obtained by moving one level down each in each component.

Our results generalise to product of $d$-many finite $b$-ary trees with a sharp critical exponent of 
\[ \alpha_{\text{crit}}(d,b) =d-1 + \log_b(b-1).\]
Geometrically, our results imply that any set $E\subset [0,1)^d$ of upper Minkowski dimension greater than $\alpha_{\text{crit}}(d,b)$ contains arithmetic $b$-adic branching patterns of arbitrary finite order.
\end{abstract}
\section{Introduction}
\subsection{Background}
In 1975 Szemer\'{e}di proved his famous theorem on arithmetic progressions in subsets of integers.
\begin{thm}[Szemer\'{e}di 1975 \cite{Sz}]\label{Th: Szm}
For every integer $k\geq 1$ and every $\delta>0$, there exists $N=N(k,\delta)$ such that whenever $n\geq N$, every subset of $\{1,\ldots,n\}$ of size at least $\delta n$ contains a non-trivial $k$-term arithmetic progression, i.e. a sequence of the form
\[a,a+d,\ldots,a+(k-1)d
\qquad\text{for some integers }a,d\geq1.
\]
\end{thm}

In 2003 Furstenberg and Weiss \cite{FW} proved a series of results extending Szemer\'{e}di's theorem to the setting of trees. Each such result roughly states that large subsets of vertices of sufficiently large trees contain arithmetically structured replicas of finite binary trees.

Let $T_n$ be the full binary tree of height $n$, realised as the set of binary words of length at most $n$ and ordered by initial segments. The level of a vertex $x$ is its word length $\labs{x}$, and the set of leaves, i.e. the terminal vertices at level $n$, is denoted by $\partial T_n\subset T_n$.
\begin{defn}\label{Def: Replica of T_k}
A \textit{replica} of $T_k$ inside a subset $H\subset T_n$ is a copy of $T_k$ in which both the level structure and the splitting structure are preserved. More precisely, it is the image of a map $\phi: T_k \to H$ such that
\begin{enumerate}[(i)]
\item for every $x,y\in T_k$ with $\labs{x}=\labs{y}$, we have that $\labs{\phi(x)} = \labs{\phi(y)}$, and
\item for every non-leaf vertex $x\in T_k$ with children $y$ and $z$, the vertices $\phi(y)$ and $\phi(z)$ are descendants of the two distinct children of $\phi(x)$, respectively.
\end{enumerate}
The sequence of levels occupied by the image $\phi(T_k)$ is called the \textit{signature} of the replica. The replica is \textit{arithmetic} if its signature is an arithmetic progression.
\end{defn}
The first result we recall from \cite{FW} is about subsets of $T_n$ which are large in the sense that they have positive density, where the density of a subset $H\subset T_n$ is defined to be
\[ d(H):=\frac{1}{n+1}\sum_{i=0}^n \frac{\labs{\{x \in H \, : \, \labs{x} = i\}}}{2^i}.\]
\begin{thm}[Furstenberg-Weiss 2003 \cite{FW}]\label{Th:FW Density}
For every $k\geq1$ and $\delta>0$, there exists $N=N(k,\delta)$ such that, for every $n\geq N$, every $H\subset T_n$ with $d(H)\geq\delta$ contains an arithmetic replica of $T_k$.
\end{thm}
By applying Theorem~\ref{Th:FW Density} to a set $H\subset T_n$ which is a union of a positive proportion of the levels of $T_n$, we recover Szemer\'{e}di's theorem.

Furstenberg and Weiss also proved a variant of Theorem~\ref{Th:FW Density} for subsets of $T_n$ which are large in quite a different sense, in that they are generated by a sufficiently large subset of the leaves $\partial T_n$. For $A\subset\partial T_n$, let
\[T_A:=\{x\in T_n:x\text{ is an initial segment of some }y\in A\}
\]
be the ancestor closure generated by $A$.
\begin{thm}[Furstenberg-Weiss 2003 \cite{FW}]\label{Th:FW Leaf}
For every $k\geq1$ and $\alpha>0$, there exists $N=N(k,\alpha)$ such that, for every $n\geq N$, if $A\subset\partial T_n$ and $\labs{A}\geq2^{\alpha n}$, then $T_A$ contains an arithmetic replica of $T_k$.
\end{thm}
The original proofs of both Theorem~\ref{Th:FW Density} and Theorem~\ref{Th:FW Leaf} used the ergodic theory of Markov processes. A simpler proof of Theorem~\ref{Th:FW Density} via purely combinatorial methods was later found by Pach, Solymosi and Tardos in \cite{PST}, where they also gave a short argument which deduced the leaf-generated result from the density result, although we stress that each of these proofs relies on either Szemer\'{e}di's theorem or its equivalent ergodic formulation in order to ensure that the replica of $T_k$ obtained is arithmetic.

In \cite{BF}, Bulinski and Fish proved the existence of certain arithmetic product-tree structures in positive-density subsets of products of trees, generalising Theorem~\ref{Th:FW Density} in such a way that their result implies the multidimensional Szemer\'{e}di's theorem, just as Theorem~\ref{Th:FW Density} implies the one-dimensional Szemer\'{e}di's theorem. The purpose of this article is to look for generalisations of the leaf-generated result to the setting of products of trees.

The precise configurations we will look for in the product setting are best motivated by the following elementary fact about leaf-generated subtrees in one dimension. For a set $A\subset \partial T_n$, we say that a non-leaf vertex $x \in T_A$ \textit{splits} whenever $T_A$ contains both children of $x$. Then as long as $A$ is non-empty we have that
\begin{equation}\label{eq: Baby splitting lemma}
T_A \text{ contains exactly }\labs{A}-1 \text{ many vertices that split.}
\end{equation}
Notice that splitting vertices in $T_A$ are exactly roots of replicas of $T_1$ in $T_A$. If we build a new leaf set $A' \subset \partial T_{n'}$ consisting of splitting vertices of $T_A$ at some common level $1\leq n'<n$, then splitting vertices in $T_{A'}$ correspond to roots of replicas of $T_2$ in $T_A$, and so on.

In just the same way that replicas are built out of nested generations of splitting vertices in one dimension, the configurations we seek in the higher dimensional setting of products of trees will be built out of nested generations of product vertices that split in an appropriate higher dimensional sense.

\subsection{Dyadic rectangles}
Our exact notion of splitting for product vertices is in fact very geometric, and this is best explained in reference to the product of two binary trees
\[ \Gamma_{n,m} = T_n \times T_m.\]
This example retains a sufficient amount of the complexity of the general case so as to demonstrate the key ideas, and more importantly it can be realised geometrically using dyadic partitions of the unit square.

The product $\Gamma_{n,m}$ is naturally identified with the collection $\mathcal{R}_{\leq n, \leq m}$ consisting of all dyadic rectangles in the unit square $[0,1)^2$ of width at least $2^{-n}$ and height at least $2^{-m}$, with the product order corresponding to containment of rectangles.

Write $\mathcal R_{i,j}$ for the set of all dyadic rectangles in $[0,1)^2$ of size $2^{-i}\times2^{-j}$ and set
\[\mathcal R_{\leq n,\leq m}:=\bigcup_{i=0}^n\bigcup_{j=0}^m\mathcal R_{i,j}.\]
A leaf set $A\subset\partial\Gamma_{n,m}$ may therefore be viewed as a collection of terminal rectangles $A\subset\mathcal R_{n,m}$, which in turn can be realised geometrically as
\[A_{\mathrm{geom}}:=\bigcup_{R\in A}R\subset[0,1)^2.\]
For $E\subset[0,1)^2$, let
\[\mathrm{occ}_{i,j}(E):=\{R\in\mathcal R_{i,j}:R\cap E\neq\emptyset\}.\]
Under the preceding identification, the ancestor closure of $A$ is then exactly
\begin{equation} \label{eq: Ancestor closure is occupied recs}
\bigcup_{i=0}^n\bigcup_{j=0}^m\mathrm{occ}_{i,j}(A_{\mathrm{geom}}).
\end{equation}
Our definition of splitting in the product tree $\Gamma_{n,m}$ can then be given in the following geometric form.
\begin{defn}
A leaf set $A\subset\mathcal R_{n,m}$ is said to \textit{split} at a dyadic rectangle $R\in\mathcal R_{i,j}$, where $i<n$ and $j<m$, if all four dyadic quadrants of $R$ are occupied by $A_{\mathrm{geom}}$. The set of all such rectangles at which $A$ splits is denoted\footnote{The subscript in $\Split_2^{(2)}(A)$ refers to dimension and the superscript refers to the fact that we are working in base $2$.} by $\Split_2^{(2)}(A)$.
\end{defn}
To search for a product analogue of \eqref{eq: Baby splitting lemma} is then to seek an answer to the following question.
\begin{qu}\label{Qu: 2d splits}
How large must a set $A\subset\mathcal R_{n,m}$ be to ensure that $\Split_2^{(2)}(A)$ is non-empty?
\end{qu}

It is not hard to see that unlike splittings in one dimension, in the product case, the size of a configuration $A\subset \mathcal{R}_{n,m}$ does not uniquely determine the size of $\Split_2^{(2)}(A)$, and so the correct analogue of fact \eqref{eq: Baby splitting lemma} must take the form of an inequality rather than an equality.

Indeed, an exact answer to Question~\ref{Qu: 2d splits} is given by the inequality
\begin{equation}\label{eq: 2d Splitting inequal}
\labs{\Split_2^{(2)}(A)} \geq \labs{A} - (2^n+2^m -1)
\end{equation}
which holds for all $A\subset \mathcal{R}_{n,m}$ and all scales $n,m\geq 1$. Moreover, the bound is sharp in the sense that there exist splitting-free leaf sets of cardinality $2^n+2^m-1$. These extremal examples are given in Section~\ref{Sec: general case} in the setting of a more general weighted splitting inequality which holds for products of $b$-ary trees.

Given a leaf set $A\subset\mathcal R_{n,m}$, the configurations we seek consist of a sequence of coordinatewise increasing scale pairs
\begin{equation}\label{eq: Signature of replica in product tree}
(i_0,j_0)<(i_1,j_1)<\cdots<(i_k,j_k)
\end{equation}
and for each $0\leq r \leq k$ a collection $\mathcal J_r\subset\mathrm{occ}_{i_r,j_r}(A_{\mathrm{geom}})$ of size $\labs{\mathcal J_r}=4^r$, such that for each $0\leq r < k$, every rectangle in $\mathcal J_r$ contains four members of $\mathcal J_{r+1}$, one in each quadrant. In the language of trees, such a configuration is called a \textit{replica of $T_k^{(4)}$ inside $\Gamma_A$}, where $T_k^{(4)}$ is the full $4$-ary tree of height $k$, and the sequence of scale pairs in \eqref{eq: Signature of replica in product tree} is called the \textit{signature} of the replica.

The same definition, with $A_{\mathrm{geom}}$ replaced by an arbitrary set $E\subset[0,1)^2$, gives a \textit{dyadic branching pattern of order $k$ in $E$}. A replica or branching pattern is arithmetic if its signature is an arithmetic progression in $\Z^2$.

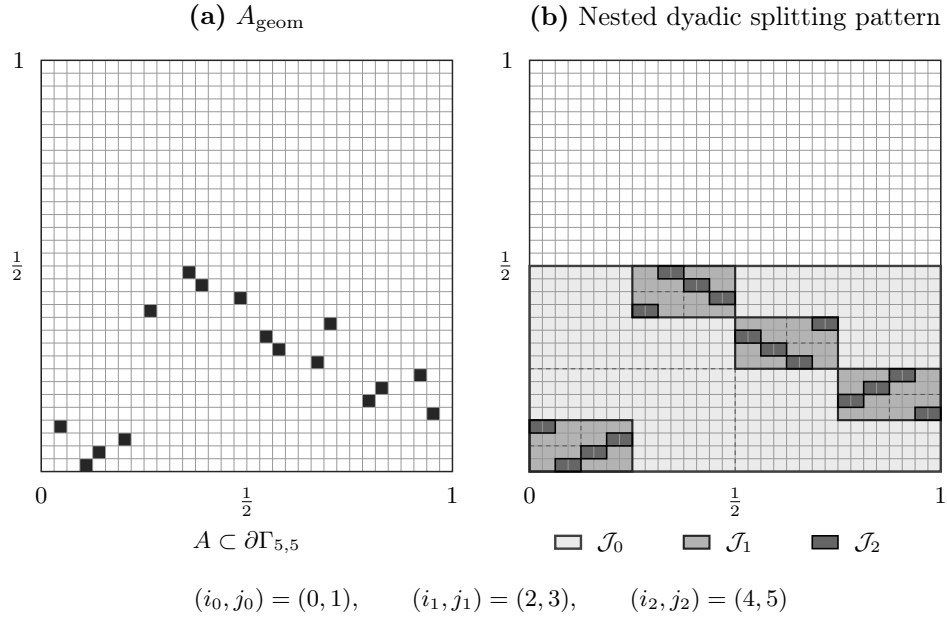
\begin{figure}[htbp]
    \centering

%
\begingroup%
\providecommand{\dyadicshowcuts}{1}%
%
\colorlet{dyadicJzeroFill}{black!8}%
\colorlet{dyadicJzeroEdge}{black!75}%
\colorlet{dyadicJoneFill}{black!30}%
\colorlet{dyadicJoneEdge}{black!85}%
\colorlet{dyadicJtwoFill}{black!60}%
\colorlet{dyadicJtwoEdge}{black}%
\begin{tikzpicture}[
    x=0.17cm, y=0.17cm,
    fine grid/.style={draw=black!40, line width=0.18pt},
    domain edge/.style={draw=black!85, line width=0.55pt},
    root box/.style={draw=dyadicJzeroEdge, line width=0.9pt},
    parent box/.style={draw=dyadicJoneEdge, line width=0.8pt},
    final box/.style={draw=dyadicJtwoEdge, line width=0.6pt},
    quadrant cut/.style={draw=black!65, line width=0.4pt,
                         dash pattern=on 1.7pt off 1.4pt},
    every node/.style={font=\footnotesize, inner sep=1pt}
]

\def\dyadicparents{
    0/0, 24/4, 8/12, 16/8}
\def\dyadicchildren{
    2/0, 4/1, 0/3, 6/2,
    24/5, 30/4, 26/6, 28/7,
    8/12, 14/13, 10/15, 12/14,
    18/9, 20/8, 16/10, 22/11}
\def\dyadicleaves{
    3/0, 4/1, 1/3, 6/2,
    25/5, 30/4, 26/6, 29/7,
    8/12, 15/13, 11/15, 12/14,
    18/9, 21/8, 17/10, 22/11}

\def\dyadicgrid{%
    \foreach \t in {0,...,32} {
        \draw[fine grid] (\t,0) -- (\t,32);
        \draw[fine grid] (0,\t) -- (32,\t);
    }
    \draw[domain edge] (0,0) rectangle (32,32);
    \node[anchor=north] at (0,-1) {$0$};
    \node[anchor=north] at (16,-1) {$\frac12$};
    \node[anchor=north] at (32,-1) {$1$};
    \node[anchor=east] at (-1,16) {$\frac12$};
    \node[anchor=east] at (-1,32) {$1$};
}

\begin{scope}
    \foreach \x/\y in \dyadicleaves
        \fill[black!85] (\x,\y) rectangle ++(1,1);
    \dyadicgrid
    \node[font=\small, anchor=south] at (16,34)
        {\textbf{(a)} $A_{\text{geom}}$};
    \node at (16,-5.5) {$A\subset\partial\Gamma_{5,5}$};
\end{scope}

\begin{scope}[shift={(38,0)}]
    \fill[dyadicJzeroFill] (0,0) rectangle (32,16);
    \foreach \x/\y in \dyadicparents
        \fill[dyadicJoneFill] (\x,\y) rectangle ++(8,4);
    \foreach \x/\y in \dyadicchildren
        \fill[dyadicJtwoFill] (\x,\y) rectangle ++(2,1);
    \dyadicgrid

    \ifnum\dyadicshowcuts=1
        \draw[quadrant cut] (16,0) -- (16,16);
        \draw[quadrant cut] (0,8) -- (32,8);
        \foreach \x/\y in \dyadicparents {
            \draw[quadrant cut] (\x+4,\y) -- (\x+4,\y+4);
            \draw[quadrant cut] (\x,\y+2) -- (\x+8,\y+2);
        }
    \fi

    \draw[root box] (0,0) rectangle (32,16);
    \foreach \x/\y in \dyadicparents
        \draw[parent box] (\x,\y) rectangle ++(8,4);
    \foreach \x/\y in \dyadicchildren
        \draw[final box] (\x,\y) rectangle ++(2,1);

    \node[font=\small, anchor=south] at (16,34)
        {\textbf{(b)} Nested dyadic splitting pattern};
    \draw[root box, fill=dyadicJzeroFill]
        (2,-6) rectangle ++(2,1);
    \node[anchor=west] at (5,-5.5) {$\mathcal J_0$};
    \draw[parent box, fill=dyadicJoneFill]
        (12,-6) rectangle ++(2,1);
    \node[anchor=west] at (15,-5.5) {$\mathcal J_1$};
    \draw[final box, fill=dyadicJtwoFill]
        (22,-6) rectangle ++(2,1);
    \node[anchor=west] at (25,-5.5) {$\mathcal J_2$};
\end{scope}

\node at (35,-10)
    {$(i_0,j_0)=(0,1),\qquad
      (i_1,j_1)=(2,3),\qquad
      (i_2,j_2)=(4,5)$};
\end{tikzpicture}%
\endgroup%
\caption{Left: $A_{\text{geom}}$ for a leaf set
$A\subset\partial\Gamma_{5,5}$.
Right: an arithmetic dyadic branching pattern of order $2$,
with $\mathcal J_0$, $\mathcal J_1$, and $\mathcal J_2$
shaded from light to dark.}

    \label{fig:dyadic-branching-pattern}
\end{figure}

Our first main result is then the following binary planar analogue of Theorem~\ref{Th:FW Leaf}.
\begin{thmAlph}\label{Theorem: Arith replicas in Binary Product Tree}
For every integer $k\geq 1$ and every $\alpha>1$, there exists $N=N(k,\alpha)$ such that, whenever $n\geq N$ and $A\subset\partial\Gamma_{n,n}$ satisfies $\labs{A}\geq2^{\alpha n}$, the ancestor closure $\Gamma_A$ contains an arithmetic replica of $T_k^{(4)}$.
\end{thmAlph}

By realising the leaf sets in Theorem~\ref{Theorem: Arith replicas in Binary Product Tree} as the set of dyadic rectangles occupied by a set $E\subset [0,1)^2$ of sufficiently large fractal dimension, we immediately obtain the following geometric consequence.

\begin{corAlph}
Every set $E\subset[0,1)^2$ with upper Minkowski dimension
\[\overline{\dim}_M(E)=\limsup_{n\to\infty}\frac{\log_2\labs{\mathrm{occ}_{n,n}(E)}}{n}>1\] contains arithmetic dyadic branching patterns of arbitrary finite order.
\end{corAlph}
We prove Theorem~\ref{Theorem: Arith replicas in Binary Product Tree} by combining our splitting inequality from \eqref{eq: 2d Splitting inequal} with one-dimensional Ramsey results for trees. The key idea is to encode $\Gamma_A$ inside a one-dimensional tree with a high branching number by grouping consecutive scales of $\Gamma_{n,m}$ into blocks. By first establishing a high-arity extension of the leaf-generated Furstenberg--Weiss theorem, we can then obtain an arithmetic one-dimensional replica of large arity inside the one-dimensional block tree that encodes $\Gamma_A$. The splitting inequality then converts this high intermediate arity into genuine local four-way product splittings inside $\Gamma_A$. A colouring version of Theorem~\ref{Th:FW Density} then synchronises the scales of these local splittings to produce the required arithmetic replica inside $\Gamma_A$.
\subsection{The general case}
The same configurations can be considered in $b$-adic partitions of
$[0,1)^d$. Dyadic rectangles are replaced by $b$-adic boxes, and their
four quadrants by the $b^d$ subboxes obtained by dividing each side
into $b$ equal parts. In tree language, the ambient object is the
product of $d$ finite $b$-ary trees
\[
\Gamma_{\mathbf n}^{(b)}
:=
T_{n_1}^{(b)}\times\cdots\times T_{n_d}^{(b)},
\qquad
\mathbf n=(n_1,\ldots,n_d)\in\Z_{\geq0}^d.
\]
As before, $\Gamma_A$ denotes the coordinatewise ancestor closure
of a terminal leaf set $A$. Replicas again consist of generations
at common coordinate levels, now branching through all $b^d$
diagonal children, and are arithmetic when their sequence of levels
forms an arithmetic progression in $\Z^d$. The precise definitions
are given in Section~\ref{Sec: general case}.
In this setting, the critical leaf-growth exponent is
\begin{equation}\label{eq: Critical growth rate}
\alpha_{\mathrm{crit}}(d,b):=d-1+\log_b(b-1).
\end{equation}

\begin{thmAlph}\label{Thm: Arith replicas in general leaf generated product trees}
Let $b\geq2$ and $d,k\geq1$ be integers, and let
$\alpha>\alpha_{\mathrm{crit}}(d,b)$.
There exists $N=N(b,d,k,\alpha)$ such that, whenever $n\geq N$ and
\[
A\subset\partial\Gamma_{(n,\ldots,n)}^{(b)},
\qquad
\labs{A}\geq b^{\alpha n},
\]
the ancestor closure $\Gamma_A$ contains an arithmetic replica of
$T_k^{(b^d)}$.
\end{thmAlph}

For $b=d=2$, the critical exponent is $1$, recovering
Theorem~\ref{Theorem: Arith replicas in Binary Product Tree}.
The general proof uses the same blocking and synchronisation
argument, with the planar splitting inequality replaced by a
weighted splitting inequality for products of $b$-ary trees.
The exponent $\alpha_{\mathrm{crit}}(d,b)$ is precisely the critical
exponential growth rate needed to force a single full product
split, and so Theorem~\ref{Thm: Arith replicas in general leaf generated product trees} shows that no larger critical exponent is needed
to force arithmetic replicas of every prescribed finite height.

The exponent is sharp. For each $n$, there are leaf sets of cardinality exactly $b^{\alpha_{\mathrm{crit}}(d,b)n}$ whose ancestor
closures do not even contain a single full product splitting. These examples are given in
Section~\ref{Sec: general case}.

As in the plane, Theorem~\ref{Thm: Arith replicas in general leaf generated product trees} has an immediate geometric consequence for fractal subsets of $[0,1)^d$ of sufficiently large dimension. The
corresponding $b$-adic branching patterns are defined just as before:
they consist of nested generations of occupied boxes at common
coordinate scales, with one chosen descendant in each immediate
$b$-adic subbox. A pattern is arithmetic when its sequence of
scales forms an arithmetic progression in $\Z^d$.
\begin{corAlph}
For all integers $b\geq2$ and $d\geq1$, every set
$E\subset[0,1)^d$ with upper Minkowski dimension
\[
\overline{\dim}_M(E)>\alpha_{\mathrm{crit}}(d,b)
\]
contains arithmetic $d$-dimensional $b$-adic branching patterns
of every finite order.
\end{corAlph}

Finally, we address the natural question of whether or not these nested splitting configurations can also be found in positive-density subsets of $\Gamma_{\mathbf n}^{(b)}$. The exact notion of density for a subset $H\subset\Gamma_{\mathbf n}^{(b)}$ is the level-averaged density
\[d(H):=\frac{1}{\prod_{j=1}^d(n_j+1)}
\sum_{\mathbf0\leq\mathbf i\leq\mathbf n}
\frac{\labs{\{x\in H:\labs{x}=\mathbf i\}}}{b^{i_1+\cdots+i_d}}\]
which is the direct analogue of the
one-dimensional density used in Theorem~\ref{Th:FW Density}.

\begin{thmAlph}\label{Th: Product density result}
For all integers $b\geq2$ and $d,k\geq1$, and every $\delta>0$,
there exists $N=N(b,d,k,\delta)$ such that, whenever
\[
\min\{n_1,\ldots,n_d\}\geq N,
\]
every $H\subset\Gamma_{\mathbf n}^{(b)}$ with $d(H)\geq\delta$
contains an arithmetic replica of $T_k^{(b^d)}$.
\end{thmAlph}

The proof of Theorem~\ref{Th: Product density result} is independent of the leaf-generated results and proceeds by first finding a suitable affine chain of product levels on which $H$ has large density, and then applying the one-dimensional density results of \cite{PST}.

\begin{ack}
The authors acknowledge substantial assistance from ChatGPT
(OpenAI) in the development of this work. Several proof strategies
and arguments were developed and refined through extended
mathematical exchanges with the system, which also assisted with
the exposition and the preparation of the TikZ figures.
Part of this work was carried out using access provided through
OpenAI's ChatGPT for Academic Researchers program.
The authors have independently checked the mathematical arguments
and references and take full responsibility for the final manuscript.
\end{ack}
Both authors were supported by the Australian Research Council through grant DP240100472.
\section{Replicas in one-dimensional trees}
For $b\geq 2$ write $\Lambda_b:=\{0,1,\ldots,b-1\}$. For any integer $n\geq 0$ the $b$-ary tree of height $n$ is the set
\[ T_n^{(b)} := \Lambda_b^{\leq n}\]
of finite words in the alphabet $\Lambda_b$ of length at most $n$, together with the ordering $\preceq$ such that $u\preceq v$ whenever $u$ is a prefix of $v$, including if $u=v$. The concatenation of two words $u,v$ is written using juxtaposition $uv$. The level of a vertex $u$ is its word length $\labs{u}$. The root of the tree is the empty word. We define the leaves and the interior of the tree to be the sets
\[\qquad  \partial T_n^{(b)} := \Lambda_b^{n},\qquad \text{and} \qquad \inte T_n^{(b)} := T_n^{(b)} \setminus \partial T_n^{(b)}\]  
respectively. For a leaf set $A\subset \partial T_n^{(b)}$ we denote by
\[ T_A = \{ u \in T_n^{(b)} \, : \, u \preceq x \text{ for some } x\in A\}\]
its ancestor closure.
\begin{defn}[Regular and arithmetic replicas]\label{Defn: Mixed arity replicas}
For $2\leq s \leq b$, a map $\phi: T_k^{(s)} \to T_n^{(b)}$ is called a \textit{regular embedding} if the following two conditions are satisfied.
\begin{enumerate}[(i)]
\item If $x,y\in T_k^{(s)}$ have $\labs{x} = \labs{y}$ then $\labs{\phi(x)} = \labs{\phi(y)}$.
\item If $x \in \inte T_k^{(s)}$ then $\phi$ maps each of the $s$ distinct children of $x$ to a descendant of a distinct child of $\phi(x)$.
\end{enumerate}
The image of $\phi$ is called a replica of $T_k^{(s)}$.
Together conditions (i) and (ii) ensure that there exists a sequence of levels
\[ 0 \leq l_0 < \cdots < l_k \leq n\]
called the signature of $\phi$ such that $\phi$ maps level $i$ of $T_k^{(s)}$ into level $l_i$ of $T_n^{(b)}$ for each $i=0,\ldots,k$. If the signature forms an arithmetic progression inside $\{0,1,\ldots,n\}$ the map and the replica are called arithmetic.
\end{defn}
\subsection{Mixed-arity replicas in leaf-generated trees}
In this subsection we shall prove the following higher-arity extension of Theorem~\ref{Th:FW Leaf}.
\begin{thm}\label{Thm: Mixed arity FW}
For $2\leq s \leq b$, $k\geq 1$ and $\alpha > \log_b(s-1)$, there exists some $N=N(b,s,k,\alpha)$ such that, for $n\geq N$, the ancestor closure of every $A\subset \partial T_n^{(b)}$ with $\labs{A} \geq b^{\alpha n}$ contains an arithmetic replica of $T_k^{(s)}$.
\end{thm}
Our main tool for proving Theorem~\ref{Thm: Mixed arity FW} is a signature-counting inequality for mixed-arity replicas in leaf-generated subtrees. In the equal-arity density setting, a similar signature-counting result was proved by Pach, Solymosi and Tardos \cite[Lemma 3']{PST}.

Our inequality will use the notion of a \textit{branching signature}. If a replica of $T_k^{(s)}$ has signature $\{l_0,\ldots,l_k\}$ according to Definition~\ref{Defn: Mixed arity replicas}, then we say it has \textit{branching signature} $\{l_0,\ldots,l_{k-1}\}$, corresponding to the levels in the image for which vertices are required to branch into $s$-many distinct children. For a leaf set $A\subset \partial T_n^{(b)}$ let
\[ \mathrm{BrSig}_s(A) \subset \mathcal{P}(\{0,1,\ldots,n-1\})\]
denote the set of all branching signatures realised by regular embeddings of finite $s$-ary trees into $T_A$. We include $\emptyset \in\mathrm{BrSig}_s(A)$ if and only if $A\neq \emptyset$, so that $\mathrm{BrSig}_s(\emptyset)= \emptyset$.
\begin{lem}\label{Lemma: Sig counting lemma}
For any $2\leq s\leq b$, any $n\geq 0$, and any $A\subset \partial T_n^{(b)}$, we have that
\begin{equation}\label{eq: branching sig inequ}
\labs{A} \leq \sum_{\sigma \in \mathrm{BrSig}_s(A)} (b-s+1)^{\labs{\sigma}} (s-1)^{n-\labs{\sigma}}.
\end{equation}
\end{lem}
\begin{proof}
We use induction on $n$. For $n=0$ the two possibilities are $A=\emptyset$ or $A=\{\emptyset\}$, and the assertions of equation~\eqref{eq: branching sig inequ} are $0 \leq 0$ and $1\leq 1$ respectively.

Now assume the lemma has been shown for $n-1$ and let $A\subset \partial T_n^{(b)}$. For each $\lambda \in \Lambda_b$ write
\[ A_\lambda: = \{ x\in \Lambda_b^{n-1} \, : \lambda x \in A\} \subset \partial T_{n-1}^{(b)}.\]
For each $\tau \subset \{0,1,\ldots,n-2\}$ let
\[ t(\tau)= \labs{\{ \lambda \in \Lambda_b \, : \, \tau \in \mathrm{BrSig}_s(A_\lambda)\}}\]
be the number of branches of $T_A$ such that $T_{A_\lambda}$ realises the branching signature $\tau$. By the inductive hypothesis applied to each $A_\lambda$ we get that
\[ \labs{A_\lambda} \leq \sum_{\tau \in \mathrm{BrSig}_s(A_\lambda)} (b-s+1)^{\labs{\tau}} (s-1)^{n-1-\labs{\tau}}.\]
Summing over $\lambda \in \Lambda_b$ yields
\[ \labs{A} \leq \sum_{\tau\subset \{0,1,\ldots,n-2\}} t(\tau) (b-s+1)^{\labs{\tau}} (s-1)^{n-1-\labs{\tau}}.\] 
For each $\tau \subset \{0,1,\ldots,n-2\}$, let $\tau^+ \subset \{1,\ldots,n-1\}$ denote the set obtained by adding $1$ to each element of $\tau$. If $t(\tau)\geq 1$ then clearly $\tau^+ \in \mathrm{BrSig}_s(A)$. If $t(\tau)\geq s$ then again $\tau^+\in\mathrm{BrSig}_s(A)$, but we also have that $\tau^+\cup\{0\} \in \mathrm{BrSig}_s(A)$ since we can join $s$-many replicas through the root. Notice there is no issue with the compatibility of the terminal vertices of these replicas since we may always assume that they lie at level $n$.

We now compare each summand with the weight assigned to the corresponding branching signature in $\mathrm{BrSig}_s(A)$ by equation~\eqref{eq: branching sig inequ}. Put
\[W_\tau:=(b-s+1)^{\labs{\tau}}(s-1)^{n-1-\labs{\tau}}.\]
Each $\tau$ with $1 \leq t(\tau) \leq s-1$ contributes at most $(s-1)W_\tau$, which is exactly the weight assigned to $\tau^+ \in \mathrm{BrSig}_s(A)$. Each $\tau$ with $s\leq t(\tau) \leq b$ contributes at most
\[bW_\tau=(b-s+1)W_\tau+(s-1)W_\tau,\]
where the two terms on the right are exactly the weights assigned to $\tau^+\cup \{0\}$ and $\tau^+$, respectively. Equation~\eqref{eq: branching sig inequ} then follows.
\end{proof}
\begin{pro}\label{Prop: Linear length 1d replicas}
For $2\leq s \leq b$, and $\alpha > \log_b(s-1)$, there exists some $c=c(s,b,\alpha)>0$ such that, for sufficiently large $n$, the ancestor closure of every $A\subset \partial T_n^{(b)}$ with $\labs{A} \geq b^{\alpha n}$ contains a replica of $T_{\floor{cn}}^{(s)}$.
\end{pro}
\begin{proof}
Let $A\subset \partial T_n^{(b)}$ be as in the statement of the proposition. If $T_A$ contains no replica of $T_k^{(s)}$, then every $\sigma \in \mathrm{BrSig}_s(A)$ has $\labs{\sigma}\leq k-1$. By Lemma~\ref{Lemma: Sig counting lemma} we must then have that
\begin{equation}\label{eq: bound for leaf set missing height k rep}
 \labs{A} \leq \sum_{r=0}^{\min(k-1,n)}\binom{n}{r} (b-s+1)^r (s-1)^{n-r}.
\end{equation}
We will show that this contradicts that $\labs{A}\geq b^{\alpha n}$ when $k = \floor{cn}$ for some appropriate choice of $c>0$.

For $x\in [0,1]$ define
\[ f(x) := h_2(x) + x \ln (b-s+1) + (1-x) \ln (s-1)\]
where $h_2(x) =-x \ln( x) - (1-x) \ln (1-x)$ is the binary entropy function, with $h_2(0):=0$ and $h_2(1):=0$. Since $\log_b(s-1)<\alpha$, we have $f(0) = \ln(s-1) < \alpha \ln (b)$. By continuity of $f$, there exist $c=c(s,b,\alpha)>0$ and $\eta>0$ such that
\begin{equation}\label{eq: f is small enough near 0}
f(x) \leq \alpha \ln(b) -\eta \qquad \text{for every } x\in [0,c].
\end{equation}
Set $k=\floor{cn}$. For $0\leq r \leq k$ the standard entropy estimate\footnote{Indeed, for $x=r/n$ the binomial theorem implies that
\[ 1= (x+(1-x))^n \geq \binom{n}{r}x^r(1-x)^{n-r}\] from which the desired bound follows.}
\[ \binom{n}{r} \leq e^{n h_2\left( r/n\right)}\]
and equation~\eqref{eq: f is small enough near 0} imply that
\begin{align*}
\binom{n}{r} (b-s+1)^r (s-1)^{n-r} &\leq e^{n h_2\left( r/n\right)}(b-s+1)^r (s-1)^{n-r}\\
&=e^{n f\left(r/n\right)} \leq e^{n(\alpha \ln(b) - \eta)} = b^{n\alpha} e^{-\eta n}.
\end{align*}
Applying this to each term in equation~\eqref{eq: bound for leaf set missing height k rep} implies that
\[ \labs{A} \leq (n+1) b^{n \alpha} e^{-\eta n},\]
but then $\labs{A} < b^{n\alpha}$ for large enough $n$ which is a contradiction.
\end{proof}
\begin{proof}[Proof of Theorem~\ref{Thm: Mixed arity FW}]
Proposition~\ref{Prop: Linear length 1d replicas} ensures that $T_A$ contains a replica of $T_{\floor{cn}}^{(s)}$ for $n$ large enough. The signature of this replica is then a subset of $\{0,1,\ldots,n\}$ of density at least $c/2$ for large enough $n$, and so it follows from Szemer\'{e}di's theorem that it must contain an arithmetic progression of length $k+1$ provided $n$ is large enough. Pruning our replica of $T_{\floor{cn}}^{(s)}$ in the obvious way then provides an arithmetic replica of $T_k^{(s)}$ inside $T_A$.
\end{proof}
\begin{rema}[Sharpness of the growth rate]
Let $W\subset \Lambda_b$ have $\labs{W} = s-1$ and put $A = W^n \subset \partial T_n^{(b)}$. Then $\labs{A} =(s-1)^n = b^{n\log_b(s-1)}$, but every non-leaf vertex in $T_A$ has exactly $s-1$ children, and so $T_A$ does not even contain a replica of $T_1^{(s)}$. This shows that the condition $\alpha > \log_b(s-1)$ in Proposition~\ref{Prop: Linear length 1d replicas} and Theorem~\ref{Thm: Mixed arity FW} is sharp.
\end{rema}
\subsection{Equal-arity replicas in finite colourings of trees}
Let us now recall the equal-arity density result of Pach, Solymosi and Tardos from which we will immediately deduce the colouring result that we require for the proof of our main theorem. We will also use their density result as part of our proof of Theorem~\ref{Th: Product density result}.

For $b\geq 2$ define the \textit{weight}\footnote{In the binary case $b=2$, the density of a subset $H\subset T_n^{(2)}$ as defined in the introduction is then $w_2(H)/(n+1)$.} of a subset $H\subset T_n^{(b)}$ to be
\[ w_b(H):= \sum_{x\in H} b^{-\labs{x}}.\]
\begin{thm}\label{Thm: Equal arity density}
For every $b\geq 2$, $\delta>0$, and $k\geq 1$, there exists $N=N(b,\delta,k)$ such that for all $n\geq N$, every $H\subset T_n^{(b)}$ with $w_b(H)\geq \delta(n+1)$ contains an arithmetic replica of $T_k^{(b)}$.
\end{thm}
Theorem~\ref{Thm: Equal arity density} is a direct consequence of \cite[Theorem 1']{PST} and Szemer\'{e}di's theorem, as described in \cite{PST} immediately after the statement of their theorem. It has the following immediate consequence for finite colourings of trees.
\begin{thm}\label{Thm: Equal arity colouring}
For all integers $b\geq 2$, $r\geq 1$ and $k\geq 1$, there exists some $N=N(b,r,k)$ such that for every $n\geq N$, every $r$-colouring $\chi :T_n^{(b)} \to \{1,\ldots,r\}$ admits a monochromatic arithmetic replica of $T_k^{(b)}$.
\end{thm}
\begin{proof}
For an $r$-colouring of $T_n^{(b)}$ with colour classes $C_1,\ldots, C_r$ we have that
\[ \sum_{i=1}^r w_b(C_i) = n+1\]
and so some colour class $C_i$ must have $w_b(C_i) \geq (n+1)/r$. Applying Theorem~\ref{Thm: Equal arity density} to this $C_i$ gives the monochromatic arithmetic replica.
\end{proof}
\section{Splittings in a product of two binary trees}
In this section we prove our splitting inequality for the product of two binary trees. The proof will serve as a model for the more general product splitting inequality for arbitrary finite products of $b$-ary trees presented in Section~\ref{Sec: general case}.

Write $T_n = T_n^{(2)}$. For integers $n,m\geq 0$, the two-dimensional binary product tree of height $(n,m)$ is the set
\[ \Gamma_{n,m} = T_n \times T_m\]
together with the product order $\preceq_2$ inherited from its component trees, where $(u,v) \preceq_2 (x,y)$ whenever $u\preceq x$ and $v\preceq y$. The level of any $(u,v) \in \Gamma_{n,m}$ is defined to be $\labs{(u,v)} := (\labs{u},\labs{v})$. The leaves and the interior of the product tree are the sets
\[ \partial \Gamma_{n,m} := \partial T_n \times \partial T_m \quad \text{and} \quad \inte \Gamma_{n,m}: = \inte T_n \times \inte T_m \]
respectively.
For any $A\subset \partial \Gamma_{n,m}$ its ancestor closure is denoted
\[ \Gamma_A=\{(u,v) \in \Gamma_{n,m} \, : \, (u,v)\preceq_2 (x,y) \text{ for some }(x,y)\in A\}.\]
\begin{defn}[Product splittings in $\Gamma_{n,m}$]
Let $A\subset \partial \Gamma_{n,m}$. We say that a vertex $(u,v) \in \Gamma_A$ \textit{splits} if
\begin{equation}\label{eq: diag children}
(u \varepsilon,v \delta) \in \Gamma_A \quad \text{for every }(\varepsilon,\delta) \in \{0,1\}^2
\end{equation}
and denote by $\Split_2^{(2)}(A)$ the set of all vertices in $\Gamma_A$ which split. The four vertices in equation~\eqref{eq: diag children} are called the \textit{diagonal children} of $(u,v)$. Notice also that $\Split_2^{(2)}(A) \subset \inte \Gamma_{n,m}$.
\end{defn}
\begin{pro}\label{Prop: 2d binary splitting count}
For all $n,m\geq 0$ and $A \subset \partial \Gamma_{n,m}$ we have that
\[ \labs{\Split_2^{(2)}(A)} \geq \labs{A} -(2^n + 2^m -1).\] 
\end{pro}
We first prove the one-dimensional fact stated in equation~\eqref{eq: Baby splitting lemma} of the introduction. For a set $C\subset \partial T_n$ write $\Split_1^{(2)}(C)$ for the set of vertices in $T_C$ with $2$ children.
\begin{lem}\label{Lemma: Baby split lem}
For all $n\geq 0$ and $C\subset \partial T_n$ we have that
\[ \labs{\Split_1^{(2)}(C)} = \labs{C} - \mathbf{1}_{C\neq \emptyset}.\]
In particular, if $C\neq \emptyset$, then $T_C$ has exactly $\labs{C}-1$ vertices with $2$ children.
\end{lem}
\begin{proof}
For $C=\emptyset$ the statement is obvious. For non-empty $C$, let $t_0$, $t_1$, and $t_2$ be the number of vertices in $T_C$ with $0$, $1$, and $2$ children respectively. By counting edges as the number of non-root vertices and as the sum of the child counts we have that
\[ t_0 + t_1 +t_2 -1 = t_1 + 2 t_2.\] 
Since $t_0 = \labs{C}$ and $t_2 = \labs{\Split_1^{(2)}(C)}$, the result follows.
\end{proof}
\begin{proof}[Proof of Proposition~\ref{Prop: 2d binary splitting count}]
Let $A\subset \partial \Gamma_{n,m}$. For each $x \in \{0,1\}^n$ let
\[ A_x:= \{y \in \{0,1\}^m \, : \, (x,y) \in A\}\]
be the vertical fibre above $x$. For each $v \in \inte T_m$ consider
\[ E_v:= \{ x \in \{0,1\}^n \, : \, v \in \Split_1^{(2)}(A_x)\} \subset \partial T_n\]
as a one-dimensional leaf set. By counting in two orders and using Lemma~\ref{Lemma: Baby split lem} we have that
\begin{align} 
\sum_{v\in \inte T_m} \labs{E_v} &= \sum_{x\in \{0,1\}^n} \labs{\Split_1^{(2)}(A_x)} \nonumber \\
&= \sum_{x\in \{0,1\}^n} \left( \labs{A_x} - \mathbf{1}_{A_x \neq \emptyset}\right)\label{eq: first use of baby split}
\geq \labs{A} - 2^n.
\end{align}
Let $u \in \Split_1^{(2)}(E_v)$ for some $v\in \inte T_m$. We claim that $(u,v) \in \Split_2^{(2)}(A)$.
Indeed, by definition of $\Split_1^{(2)}(E_v)$, for each $\varepsilon \in \{0,1\}$ there exists some $x_\varepsilon \in E_v$ such that $u\varepsilon \preceq x_\varepsilon$. By definition of $E_v$, for each $\varepsilon \in \{0,1\}$, $v \in \Split_1^{(2)}(A_{x_\varepsilon})$, and so for each $\delta \in \{0,1\}$ there exists some $y_{\delta}^{\varepsilon} \in A_{x_\varepsilon}$ such that $v\delta \preceq y_{\delta}^{\varepsilon}$. It then follows that $(u\varepsilon,v\delta) \preceq_2 (x_\varepsilon,y_{\delta}^\varepsilon) \in A$ for each $(\varepsilon,\delta )\in \{0,1\}^2$, which proves the claim.

It follows that
\begin{align*}
\labs{\Split_2^{(2)}(A)} &\geq \sum_{v\in  \inte T_m} \labs{\Split_1^{(2)}(E_v)}\\
&\geq\sum_{v\in \inte T_m} \left(\labs{E_v}-1\right)\\
&\geq \labs{A}-2^n - (1+2+\cdots+2^{m-1}) = \labs{A} - (2^n + 2^m-1).
\end{align*}
where in the second line we use Lemma~\ref{Lemma: Baby split lem} again, noting that the bound $\labs{\Split_1^{(2)}(E_v) }\geq \labs{E_v}-1$ holds irrespective of whether $E_v$ is empty or not, and in the last line we use equation~\eqref{eq: first use of baby split}.
\end{proof}
\section{Encoding a product of two binary trees by blocks}
In this section we describe the exact block encoding we use in the proof of Theorem~\ref{Theorem: Arith replicas in Binary Product Tree} and establish some of its relevant properties.

Fix integers $q\geq 1$ and $L\geq 1$, and let $B=4^q$. Set 
\[\Omega_q:= \{0,1\}^q \times \{0,1\}^q\]
and denote by $\Omega_q^{\leq L}$ the height $L$ tree in the $B$-many \textit{block letters} $(\alpha,\beta)\in \Omega_q$.  We will refer to $\Omega_q$ as the \textit{block alphabet} and to $\Omega_q^{\leq L}$ as the \textit{block tree}. As a tree, $\Omega_q^{\leq L}$ is isomorphic to $T_L^{(B)}$, however, we use the notation $\Omega_q^{\leq L}$ to emphasise that the alphabet is $\Omega_q$. For a leaf set $C \subset \Omega_q^L$, we denote by $T_C \subset \Omega_q^{\leq L}$ its ancestor closure in the block tree.

Define the block map $\Theta_q: \Omega_q^{\leq L} \to \Gamma_{qL,qL}$ by
\begin{equation}\label{eq: Block map formula}
\Theta_q\big( (\alpha_1,\beta_1)\ldots(\alpha_t,\beta_t)\big) = (\alpha_1 \ldots \alpha_t,\beta_1 \ldots \beta_t)
\end{equation}
where each $(\alpha_i,\beta_i) \in \Omega_q$ is a block letter. It is easy to see that $\Theta_q$ is an order isomorphism from $\Omega_q^{\leq L}$ to the set
\begin{equation}\label{eq: diag block levels}
\{(u,v) \in \Gamma_{qL,qL} \, : \, \labs{(u,v)} = (tq,tq) \text{ for some }0\leq t \leq L\}\subset \Gamma_{qL,qL}.
\end{equation}
By this we mean that if $\omega,\omega' \in \Omega_q^{\leq L}$ have $\omega \preceq \omega'$, then $\Theta_q(\omega)\preceq_2 \Theta_q(\omega')$, and conversely, if $(u,v)$ and $(x,y)$ belong to the set in equation~\eqref{eq: diag block levels} and satisfy that $(u,v) \preceq_2 (x,y)$, then $\Theta_q^{-1}((u,v)) \preceq \Theta_q^{-1}((x,y))$.

For the remainder of this section, fix some $A\subset \partial \Gamma_{qL,qL}$ and let
\[\tilde{A} := \Theta_q^{-1}(A)\subset  \Omega_{q}^L\]
be its pullback under the block map. Since $\Theta_q$ is an order isomorphism, the ancestor closure $T_{\tilde{A}}$ of $\tilde{A}$ satisfies $\Theta_q(T_{\tilde{A}}) \subset \Gamma_A$.

The key mechanism behind our proof of Theorem~\ref{Theorem: Arith replicas in Binary Product Tree} is the following observation: if a vertex $\omega \in T_{\tilde{A}}$ has at least $2^{q+1}$ children, then Proposition~\ref{Prop: 2d binary splitting count} ensures that its image $\Theta_q(\omega) \in \Gamma_A$ must have a nearby descendant which is a genuine product splitting vertex in $\Gamma_A$. This observation is made precise by the following lemma.
\begin{lem}\label{Lemma: Block branching gives local splitting}
Let $\omega \in T_{\tilde{A}}$ and write $\Theta_q(\omega)=(U,V) \in \Gamma_A$. For any set 
\[ C\subset \{\lambda \in \Omega_q \, : \, \omega \lambda \in T_{\tilde{A}}\}\] 
with $\labs{C}\geq 2^{q+1}$ there exists some $(u,v) \in \{0,1\}^{\leq q-1}\times \{0,1\}^{\leq q-1}$ such that $(Uu,Vv) \in \Split_2^{(2)}(A)$. Moreover, $(u,v)$ has the additional property that for each $(\varepsilon,\delta) \in \{0,1\}^2$ there exists some $\lambda_{\varepsilon,\delta} \in C$ with
\[ (Uu\varepsilon,V v \delta) \preceq_2 \Theta_q(\omega \lambda_{\varepsilon,\delta}).\] 
\end{lem}
\begin{proof}
Identify $\Omega_q = \partial \Gamma_{q,q}$ so that $C$ is a subset of the leaves of $\Gamma_{q,q}$. Since $\labs{C}\geq 2^{q+1}$, Proposition~\ref{Prop: 2d binary splitting count} ensures that there exists some $(u,v) \in \Split_2^{(2)}(C)$.
By definition of $\Split_2^{(2)}(C)$, for every $(\varepsilon,\delta)\in \{0,1\}^2$ there exists $(\alpha_{\varepsilon,\delta},\beta_{\varepsilon,\delta}) \in C$ with
\begin{equation}\label{eq: extending local splits}
(u\varepsilon,v\delta) \preceq_2 (\alpha_{\varepsilon,\delta},\beta_{\varepsilon,\delta})
\end{equation}
in $\Gamma_C$, which implies that
\[ (Uu \varepsilon,V v \delta) \preceq_2 (U\alpha_{\varepsilon,\delta},V\beta_{\varepsilon,\delta})\]
in $\Gamma_{qL,qL}$. If we can ensure that each
\[(U\alpha_{\varepsilon,\delta},V\beta_{\varepsilon,\delta}) \in \Gamma_A\]
then we will have that $(Uu,Vv) \in \Split_2^{(2)}(A)$. This is immediate since
\[ (U\alpha_{\varepsilon,\delta},V\beta_{\varepsilon,\delta}) = \Theta_q(\omega (\alpha_{\varepsilon,\delta},\beta_{\varepsilon,\delta}))\] 
and by definition of $C$, each $\omega (\alpha_{\varepsilon,\delta},\beta_{\varepsilon,\delta}) \in T_{\tilde{A}}$. Taking $\lambda_{\varepsilon,\delta} = (\alpha_{\varepsilon,\delta},\beta_{\varepsilon,\delta})$ finishes the proof.
\end{proof}
We emphasise that the splitting vertex $(Uu,Vv) \in \Split_2^{(2)}(A)$ produced by Lemma~\ref{Lemma: Block branching gives local splitting} occurs within the next block, within at most $q-1$ additional levels down in each coordinate. Moreover, the relative depths $\labs{u}$ and $\labs{v}$ may vary across different vertices.

Nevertheless, if $T_{\tilde{A}}$ contains a replica of the $2^{q+1}$-ary tree, then the preceding lemma will allow us to produce an embedding of the $4$-ary tree inside $\Gamma_A$ in which the $4$-way branching will map through the $4$ diagonal children of one of these local product splittings $(Uu,Vv)$. This embedding will be \textit{nearly-regular}, in the sense that the only deviation from regularity will come from the bounded variations in the relative depths of $\labs{u}$ and $\labs{v}$.
Similarly, if the replica of the $2^{q+1}$-ary tree inside $T_{\tilde{A}}$ is arithmetic, then the resulting embedding of the $4$-ary tree inside $\Gamma_A$ will be \textit{nearly-arithmetic}. The final part of the proof of Theorem \ref{Theorem: Arith replicas in Binary Product Tree} will then use a colouring argument to turn this \textit{nearly-arithmetic} and \textit{nearly-regular} embedding into a genuine arithmetic replica of $T_k^{(4)}$.
\begin{lem}\label{Lem: S-ary tree in block tree gives nearly arithmetic 4-ary tree}
Let $S:=2^{q+1}$ and $N\geq 0$. Suppose that
\[ \Phi: T_{N+1}^{(S)} \to T_{\tilde{A}}\]
is an arithmetic regular embedding. Then there exists a map
\[ G: T_N^{(4)} \to \Split_2^{(2)}(A)\subset \Gamma_A\]
which is \textit{nearly-arithmetic} and \textit{nearly-regular}, in the following precise sense.
\begin{enumerate}[(i)]
\item For each non-leaf vertex $z \in T_N^{(4)}$, each of its four distinct children are mapped under $G$ to a descendant of a distinct diagonal child of $G(z)$.
\item There exist integers $t_0\geq 0$ and $\Delta\geq 1$ such that for each $z \in T_N^{(4)}$ we have that
\[ \labs{G(z)} = q(t_0+\labs{z}\Delta)(1,1)+(a_z,b_z) \] 
for some $a_z,b_z \in \{0,1,\ldots,q-1\}$.
\end{enumerate}
\end{lem}
\begin{proof}
We will build the map $G$ inductively. To each vertex $z \in T_N^{(4)}$ we will associate a vertex $x_z \in T_{N+1}^{(S)}$ with $\labs{z} = \labs{x_z}$. The image $G(z)$ will be one of the local product splittings below $\Theta_q(\Phi(x_z))$ whose existence is guaranteed by Lemma~\ref{Lemma: Block branching gives local splitting}.

Start by associating the root of $T_N^{(4)}$ with the root of $T_{N+1}^{(S)}$. Suppose that $z$ has been associated with some $x_z$, and write
\[ \Theta_q(\Phi(x_z)) = (U_z,V_z).\]
Since $\labs{x_z} = \labs{z} \leq N$, the vertex $x_z$ has $S$ children in $T_{N+1}^{(S)}$. As $\Phi$ is a regular embedding, it maps the $S$ children of $x_z$ below $S$ many distinct children of $\Phi(x_z)$. Let
\[ C_{x_z}=\{ \lambda \in \Omega_q \, : \, \Phi(x_z) \lambda \preceq \Phi(c) \text{ for some child } c \text{ of }x_z\}\]
be the set of immediate block labels under which $\Phi$ maps the children of $x_z$. Thus $\labs{C_{x_z}}=S=2^{q+1}$, and every corresponding block child $\Phi(x_z)\lambda$ belongs to $T_{\tilde{A}}$, since it is an ancestor of some $\Phi(c)$.

Apply Lemma~\ref{Lemma: Block branching gives local splitting} at $\Phi(x_z)$ using the selected labels $C_{x_z}$ to obtain a pair of words $(u_z,v_z) \in \Gamma_{q,q}$ with
\[a_z:=\labs{u_z} < q \qquad \text{ and } \qquad b_z:= \labs{v_z}<q\]
such that
\[ G(z):= (U_z u_z,V_z v_z) \in \Split_{2}^{(2)}(A).\]
The additional property in Lemma~\ref{Lemma: Block branching gives local splitting} ensures that for each $(\varepsilon,\delta)\in \{0,1\}^2$ there exists some $\lambda_{\varepsilon,\delta} \in C_{x_z}$ for which
\begin{equation}\label{eq: G satisfies property (i)}
(U_z u_z \varepsilon,V_z v_z \delta) \preceq_2 \Theta_q(\Phi(x_z \lambda_{\varepsilon,\delta}))
\end{equation}
and so if $\labs{z} < N$, we can continue the construction by associating each of the four distinct children of $z$ to one of the four distinct children of $x_z$ of the form $x_z \lambda_{\varepsilon,\delta}$. At the final generation, when $\labs{z} = N$, the additional generation in $T_{N+1}^{(S)}$ ensures that $x_z$ still has $S$ children required to supply the local splitting.

To verify property (i), let $z'$ be the child of $z$ associated with $x_z \lambda_{\varepsilon,\delta}$. By construction, $x_{z'} = x_z \lambda_{\varepsilon,\delta}$, and $G(z')$ is a descendant of $\Theta_q(\Phi(x_{z'}))$, so by equation~\eqref{eq: G satisfies property (i)} we have that
\[ (U_z u_z \varepsilon,V_z v_z \delta) \preceq_2 \Theta_q(\Phi(x_{z'})) \preceq_2 G(z').\]
That is, the four children of $z$ are mapped below the four distinct diagonal children of $G(z)$.

Finally, since $\Phi$ is arithmetic there exist $t_0\geq 0$ and $\Delta\geq 1$ such that its signature is of the form
\[ t_0, t_0 + \Delta, \ldots, t_0 + N\Delta.\] 
Each $z\in T_N^{(4)}$ has that $\labs{x_z} = \labs{z}$ and so
\[ \labs{U_z} = \labs{V_z} = q(t_0 + \labs{z} \Delta).\]
It then follows that
\[ \labs{G(z)} = \labs{(U_z u_z,V_z v_z)} = q(t_0 + \labs{z} \Delta)(1,1) + (a_z,b_z),\] 
which proves (ii).
\end{proof}

\section{Arithmetic replicas in a product of two binary trees}
We are now ready to show that large leaf-generated subtrees of a product of two binary trees contain arithmetic replicas of $T_k^{(4)}$ as defined in the introduction. We recall the exact definition.
\begin{defn}[Replicas of $T_k^{(4)}$ in the binary product tree]
A map $\phi: T_k^{(4)} \to \Gamma_{n,m}$ is called a \textit{regular embedding} if it satisfies the following two conditions.
\begin{enumerate}[(i)]
\item If $x,y\in T_k^{(4)}$ have $\labs{x} = \labs{y}$ then $\labs{\phi(x)} = \labs{\phi(y)}$.
\item If $x\in \inte T_k^{(4)}$ then $\phi$ maps each of the four distinct children of $x$ to a descendant of a distinct diagonal child of $\phi(x)$.
\end{enumerate} 
The image of $\phi$ is called a replica of $T_k^{(4)}$. Together conditions (i) and (ii) ensure that there exists a sequence of levels
\[  (i_0,j_0) < \cdots < (i_k,j_k) \]
strictly increasing in each coordinate called the \textit{signature of $\phi$} such that $\phi$ maps level $t$ of $T_k^{(4)}$ into level $(i_t,j_t)$ of $\Gamma_{n,m}$. The embedding and replica are called \textit{arithmetic} if there exists $(a,b)\in \Z_{\geq 1}^2$ such that the signature is of the form
\[ (i_t,j_t) = (i_0,j_0) + t (a,b) \qquad \text{for each } t=0,1,\ldots, k.\]
\end{defn}
The following is a refined form of Theorem~\ref{Theorem: Arith replicas in Binary Product Tree}, in which we additionally show that the common difference of the signature can be chosen in the diagonal direction.
\begin{thm}[Refined form of Theorem~\ref{Theorem: Arith replicas in Binary Product Tree}]\label{Thm: Refined binary product theorem}
For every $\alpha>1$ and $k\geq 1$, there exists $N_0 = N_0(\alpha,k)$ such that whenever $n\geq N_0$, every 
\[ A\subset \partial \Gamma_{n,n}\qquad \text{with} \qquad \labs{A}\geq 2^{\alpha n}\]
has an ancestor closure $\Gamma_A$ containing an arithmetic replica of $T_k^{(4)}$. Moreover, the replica can be chosen so that its signature has the form
\[ (i_0,j_0) + t (Q,Q) \qquad \text{for each } t=0,1,\ldots, k\]
for some $Q\geq 1$.
\end{thm}
\begin{proof}
First choose an integer $q=q(\alpha)$ large enough in terms of $\alpha$ so that
\begin{equation}
\frac{\log_2(2^{q+1}-1)}{q} < \alpha.
\end{equation} 
Put
\[ \Omega_q:= \{0,1\}^q \times \{0,1\}^q, \qquad B: = 4^q, \qquad \text{and}\quad S:=2^{q+1}.\]
Then $2\leq S \leq B$, and our choice of $q$ implies that
\begin{equation}\label{eq: condition letting us pick gamma}
 \log_B(S-1) < \frac{\alpha}{2}.
 \end{equation}

By Theorem~\ref{Thm: Equal arity colouring} we can pick some $N=N(\alpha,k)$ large enough so that every $q^2$-colouring of $T_N^{(4)}$ contains a monochromatic arithmetic replica of $T_k^{(4)}$.

Write
\[ n =qL + r \qquad \text{ for some } 0\leq r < q,\]
and let
\[ A':= \{(x,y) \in \partial\Gamma_{qL,qL} \, : \, (x,y)\preceq_2 (u,v) \text{ for some }(u,v) \in A\}.\]
Since $\Gamma_{A'} = \Gamma_A\cap \Gamma_{qL,qL}$, it suffices to show that $\Gamma_{A'}$ contains an arithmetic replica of $T_k^{(4)}$ for $n$ large enough.
There are at most $2^{2r}$ possible extensions of any $(x,y)\in A'$ to a leaf in $A$, so
\begin{equation}\label{eq: dealing with harmless truncation}
\labs{A'} \geq 2^{-2r} \labs{A} \geq 2^{\alpha (qL+r) - 2r}.
\end{equation}
Let $\Theta_q: \Omega_q^{\leq L} \to \Gamma_{qL,qL}$ be the block map from equation~\eqref{eq: Block map formula} and set
\[ \tilde{A}:= \Theta_q^{-1}(A') \subset \Omega_q^{L}.\] 
Since $\Theta_q$ is a bijection onto the $q$-synchronous levels of $\Gamma_{qL,qL}$, we have $|\tilde{A}| = \labs{A'}$, and so equation~\eqref{eq: dealing with harmless truncation} gives
\[ \log_B\left(|\tilde{A}|\right) \geq \frac{\alpha}{2}L + \frac{(\alpha-2)r}{2q}.\]
The second term is bounded independently of $L$, so by equation \eqref{eq: condition letting us pick gamma} there exists some $\gamma \in (\log_B(S-1),\alpha/2)$ such that
\begin{equation}\label{eq: leaves in block tree are big enough}
 |\tilde{A}| \geq B^{\gamma L}
\end{equation}
for sufficiently large $n$. Since the block tree $\Omega_q^{\leq L}$ is isomorphic to a $B$-ary tree and $\gamma > \log_B(S-1)$, equation~\eqref{eq: leaves in block tree are big enough} ensures, for $n$ large enough, that $|\tilde{A}|$ is large enough for Theorem~\ref{Thm: Mixed arity FW} to guarantee the existence of an arithmetic regular embedding
\[ \Phi: T_{N+1}^{(S)} \to T_{\tilde{A}}.\]
Let $\Phi$ have signature
\[ \{t_0,t_0 + \Delta, \ldots, t_0 + (N+1)\Delta\}\]
for some $\Delta\geq 1$. By Lemma~\ref{Lem: S-ary tree in block tree gives nearly arithmetic 4-ary tree} there exists a map
\[ G: T_N^{(4)} \to \Split_2^{(2)}(A') \subset \Gamma_A\]
which preserves the four-way branching structure (through the diagonal children of each image vertex) and satisfies that for each $z \in T_N^{(4)}$,
\[ \labs{G(z)} = q(t_0 + \labs{z}\Delta,t_0 + \labs{z}\Delta) + (a_z,b_z)\]
for some $(a_z,b_z)\in \{0,1,\ldots,q-1\}^2$. The assignment
\[ z \mapsto (a_z,b_z)\]
gives a $q^2$-colouring of $T_N^{(4)}$, so by our choice of $N$ there exists a monochromatic arithmetic regular embedding
\[ f: T_k^{(4)} \to T_N^{(4)}.\]
Let $f$ have signature
\[ \{r_0,r_0 +D, \ldots, r_0 + kD\}\]
for some $D\geq 1$ and let its common colour be $(a,b) \in \{0,1,\ldots,q-1\}^2$. The map
\[ F:= G\circ f: T_k^{(4)} \to \Gamma_A\]
is then an arithmetic regular embedding of $T_k^{(4)}$. Indeed, for every $z\in T_k^{(4)}$ we can calculate
\begin{align*}
\labs{F(z)} = \labs{G(f(z))} &= q\left(t_0 + \labs{f(z)}\Delta,t_0 + \labs{f(z)}\Delta\right) + (a,b)\\
& = \underbrace{q(t_0 + r_0 \Delta)(1,1) + (a,b)}_{:=\mathbf{v}} + \labs{z} \underbrace{q D \Delta }_{:=Q}(1,1),
\end{align*}
and so $F$ has signature
\[ \mathbf{v}, \mathbf{v} + Q(1,1),\ldots, \mathbf{v} + kQ (1,1)\]
as claimed.
\end{proof}
\section{Leaf-generated results for \texorpdfstring{$d$}{d}-fold products of \texorpdfstring{$b$}{b}-ary trees}\label{Sec: general case}
For $d\geq1$, $b\geq2$, and $\mathbf n=(n_1,\ldots,n_d)\in \Z_{\geq 0}^d$ the $d$-fold $b$-ary product tree of height $\mathbf{n}$ is the set
\[\Gamma_{\mathbf n}^{(b)}:=T_{n_1}^{(b)}\times\cdots\times T_{n_d}^{(b)}\]
with the product ordering $\preceq_d$ where $(u_1,\ldots,u_d) \preceq_d (v_1,\ldots,v_d)$ whenever $u_i \preceq v_i$ for each $i=1,\ldots,d$. The level of any $(u_1,\ldots,u_d)\in \Gamma_{\mathbf{n}}^{(b)}$ is defined to be $\labs{(u_1,\ldots,u_d)} := (\labs{u_1},\ldots,\labs{u_d})$. The leaves and the interior of the product tree are the sets
\[ \partial \Gamma_{\mathbf{n}}^{(b)}: = \partial T_{n_1}^{(b)}\times \cdots \times \partial T_{n_d}^{(b)} \]
and
\[ \inte \Gamma_{\mathbf{n}}^{(b)}: = \inte T_{n_1}^{(b)} \times \cdots \times \inte T_{n_d}^{(b)}\]
respectively. For any $A\subset \partial \Gamma_{\mathbf{n}}^{(b)}$ its ancestor closure is denoted
\[ \Gamma_A = \{ \mathbf{u} \in \Gamma_{\mathbf{n}}^{(b)} \, : \, \mathbf{u} \preceq_d \mathbf{x} \text{ for some } \mathbf{x} \in A\}.\]  
For any $\mathbf{u}=(u_1,\ldots,u_d), \mathbf{v} = (v_1,\ldots,v_d) \in \Gamma_{\mathbf{n}}^{(b)}$ we write
\[ \mathbf{u}\mathbf{v} = (u_1 v_1,\ldots,u_d v_d).\]
\subsection{Splittings in leaf-generated subsets of \texorpdfstring{$\Gamma_{\mathbf{n}}^{(b)}$}{Gamma(n,b)}}
In this subsection we prove our splitting inequality in the general product case.
\begin{defn}
Let $A\subset \partial \Gamma_{\mathbf{n}}^{(b)}$. An interior vertex $\mathbf{u} \in \Gamma_A$ \textit{splits} if
\begin{equation}\label{eq: d dim diag children}
\mathbf{u}\bm{\epsilon} \in \Gamma_A \quad \text{ for every }\bm{\epsilon} \in \Lambda_b^d.
\end{equation}
The $b^d$ vertices in equation~\eqref{eq: d dim diag children} are called the \textit{diagonal children} of $\mathbf{u}$. We denote by $\Split_d^{(b)}(A)$ the set of all vertices in $\Gamma_A$ that split.
\end{defn}
Once we move past the binary case of $b=2$, our splitting inequality no longer counts the size of $\Split_d^{(b)}(A)$ directly, but rather will be stated in terms of a weighted splitting count defined for any $A\subset \partial \Gamma_{\mathbf{n}}^{(b)}$ to be
\[ \mathcal{S}_{d,b}^{\mathbf{n}}(A): = \sum_{\mathbf{u} \in \Split_d^{(b)}(A)} \prod_{i=1}^{d} \left(b-1\right)^{n_i - \labs{u_i}-1}.\]
Notice that for $b=2$ we have that $\mathcal{S}_{d,2}^{\mathbf{n}}(A) = \labs{\Split_d^{(2)}(A)}$. In any case, our splitting inequality will ensure that $\Split_d^{(b)}(A)$ is non-empty whenever $A\subset \partial \Gamma_{\mathbf{n}}^{(b)}$ has cardinality strictly larger than
\begin{equation}\label{eq: general cutoff}
C_{d,b}(\mathbf{n}): = \prod_{i=1}^d b^{n_i} - \prod_{i=1}^{d}\left(b^{n_i} - (b-1)^{n_i}\right).\end{equation}
\begin{pro}\label{Prop: General splitting lemma}
For every $b\geq 2$, $d\geq 1$, $\mathbf{n}\in \Z_{\geq 0}^d$ and $A\subset \partial \Gamma_{\mathbf{n}}^{(b)}$ we have that
\[ \mathcal{S}_{d,b}^{\mathbf{n}}(A) \geq \labs{A} - C_{d,b}(\mathbf{n}).\]
In particular, if $\labs{A}> C_{d,b}(\mathbf{n})$, then $\Split_d^{(b)}(A) \neq \emptyset$.
\end{pro}
We will prove Proposition~\ref{Prop: General splitting lemma} by induction on the dimension $d$ in much the same way that we proved Proposition~\ref{Prop: 2d binary splitting count}. We first establish the base case of $d=1$.
\begin{lem}\label{Lem: 1d b-ary base case}
For every $n\geq 0$, $b\geq 2$, and $A\subset \partial T_n^{(b)}$ we have that
\[ \sum_{u \in \Split_1^{(b)}(A)} (b-1)^{n-\labs{u}-1} \geq \labs{A} -(b-1)^n.\] 
\end{lem}
\begin{proof}
If $n=0$ the result is immediate so suppose that $n\geq 1$. Let $a_i$ be the number of vertices in $T_A$ at level $i$, and let $f_i$ be the number of splitting vertices at level $i$. A splitting vertex has $b$ children and a non-splitting vertex has at most $b-1$ children and so for any $0\leq i <n$ we have that
\[ a_{i+1} \leq b f_i + (b-1)(a_i-f_i) = (b-1)a_i + f_i,\]
or equivalently
\[ f_i \geq a_{i+1} - (b-1) a_i.\] 
Multiplying by $(b-1)^{n-i-1}$ and summing over levels gives
\begin{align*}
\sum_{u\in \Split_1^{(b)}(A)} (b-1)^{n-\labs{u}-1} & = \sum_{i=0}^{n-1}f_i (b-1)^{n-i-1}\\
& \geq \sum_{i=0}^{n-1} \left[ (b-1)^{n-i-1}a_{i+1} - (b-1)^{n-i}a_i \right]\\
&= a_n - (b-1)^n a_0 \geq \labs{A} - (b-1)^n
\end{align*}
since the final sum telescopes.
\end{proof}
\begin{proof}[Proof of Proposition~\ref{Prop: General splitting lemma}]
We proceed by induction on $d$. The case $d=1$ is exactly the previous lemma, since $C_{1,b}(n)=(b-1)^n$.

If any coordinate height $n_i=0$, then no split is possible and $C_{d,b}(\mathbf{n}) = \prod_{i=1}^d b^{n_i}$, since the second term in equation~\eqref{eq: general cutoff} vanishes. Thus the desired inequality is trivial.

So let $d\geq 2$ and $\mathbf{n} =(\mathbf{n}',m) \in \Z_{\geq 1}^d$ where $\mathbf{n}'\in \Z_{\geq 1}^{d-1}$ are the first $d-1$ components of $\mathbf{n}$. Assume the result has been proved for $d-1$, and let $A\subset \partial \Gamma_{\mathbf{n}}^{(b)}$. For each $\mathbf{x}\in \partial \Gamma_{\mathbf{n}'}^{(b)}$ let
\[ A_{\mathbf{x}}:= \{ y \in \partial T_m^{(b)} \, : \, (\mathbf{x},y)\in A\}\]
be the $d^\text{th}$ coordinate fibre of $A$ above $\mathbf{x}$, and for each $v \in \inte T_m^{(b)}$ let
\[ E_v:=\{\mathbf{x} \in \partial \Gamma_{\mathbf{n}'}^{(b)} \, : \, v \in \Split_{1}^{(b)}(A_{\mathbf{x}})\}.\]
By counting in two orders and using Lemma~\ref{Lem: 1d b-ary base case} we have that
\begin{align}
\sum_{v\in \inte T_m^{(b)}} (b-1)^{m-\labs{v}-1}\labs{E_v} &= \sum_{\mathbf{x}\in \partial \Gamma_{\mathbf{n}'}^{(b)}} \sum_{v \in \Split_{1}^{(b)}(A_\mathbf{x})} (b-1)^{m-\labs{v}-1} \nonumber \\
& \geq \sum_{\mathbf{x}\in \partial \Gamma_{\mathbf{n}'}^{(b)}} \left(\labs{A_{\mathbf{x}}} - (b-1)^m \right) \nonumber \\
&= \labs{A} - \labs{\partial \Gamma_{\mathbf{n}'}^{(b)}}(b-1)^m. \label{eq: general split eq 1}
\end{align}
If $\mathbf{u} \in \Split_{d-1}^{(b)}(E_v)$, then we claim that $(\mathbf{u},v) \in \Split_{d}^{(b)}(A)$. Indeed, for each $\bm{\epsilon} \in \Lambda_{b}^{d-1}$, the splitting of $E_v$ gives some $\mathbf{x}_{\bm{\epsilon}} \in E_v$ with $\mathbf{u}\bm{\epsilon} \preceq_{d-1} \mathbf{x}_{\bm{\epsilon}}$. For each such $\mathbf{x}_{\bm{\epsilon}}$, $v\in \Split_{1}^{(b)}(A_{\mathbf{x}_{\bm{\epsilon}}})$, and so for each $\delta \in \Lambda_b$ there exists some $(\mathbf{x}_{\bm{\epsilon}},y_\delta^{\bm{\epsilon}}) \in A$ with $(\mathbf{u} \bm{\epsilon},v \delta)\preceq_d (\mathbf{x}_{\bm{\epsilon}},y_\delta^{\bm{\epsilon}})$.

It follows that
\begin{align*}
\mathcal{S}_{d,b}^\mathbf{n}(A)&=  \sum_{(\mathbf{u},v) \in \Split_d^{(b)}(A)} \prod_{i=1}^{d-1} \left(b-1\right)^{n_i - \labs{u_i}-1}(b-1)^{m-\labs{v}-1}\\
&\geq \sum_{v \in \inte T_m^{(b)}}\sum_{\mathbf{u} \in \Split_{d-1}^{(b)}(E_v)}\prod_{i=1}^{d-1} \left(b-1\right)^{n_i - \labs{u_i}-1}(b-1)^{m-\labs{v}-1}\\
&= \sum_{v \in \inte T_m^{(b)}} \mathcal{S}_{d-1,b}^{\mathbf{n}'}(E_v) (b-1)^{m-\labs{v}-1}\\
& \geq \sum_{v \in \inte T_m^{(b)}} \left( \labs{E_v} - C_{d-1,b}(\mathbf{n}')\right) (b-1)^{m-\labs{v}-1}\\
&\geq \labs{A} - \labs{\partial \Gamma_{\mathbf{n}'}^{(b)}}(b-1)^m - C_{d-1,b}(\mathbf{n}') \sum_{v\in \inte T_m^{(b)}} (b-1)^{m-\labs{v}-1}\\
& = \labs{A} - C_{d,b}(\mathbf{n}),
\end{align*}
where in the second inequality we use the inductive hypothesis, in the third inequality we use equation~\eqref{eq: general split eq 1}. The final equality follows via direct computation using the definition of $C_{d,b}(\mathbf{n})$ since we may use geometric summation to calculate
\[ \sum_{v\in \inte T_m^{(b)}} (b-1)^{m-\labs{v}-1} = \sum_{i=0}^{m-1} b^i (b-1)^{m-i-1} = b^m - (b-1)^m.\]
\end{proof}
\subsection{Arithmetic replicas in \texorpdfstring{$\Gamma_{\mathbf{n}}^{(b)}$}{Gamma(n,b)}}
With Proposition~\ref{Prop: General splitting lemma} in hand, we have all the tools needed to prove that sufficiently large leaf-generated subsets of $\Gamma_{\mathbf{n}}^{(b)}$ admit arithmetically structured regular embeddings of $T_k^{\left(b^d\right)}$ in which the $b^d$-ary branching structure of $T_k^{\left(b^d\right)}$ embeds through all $b^d$ diagonal children of each image vertex, i.e. through a genuine $b$-ary $d$-dimensional product splitting.

\begin{defn}[Replicas of $T_k^{\left(b^d\right)}$ in $\Gamma_{\mathbf{n}}^{(b)}$]
A map $\phi : T_k^{\left(b^d \right)} \to \Gamma_{\mathbf{n}}^{(b)}$ is called a \textit{regular embedding} if it satisfies the following two conditions.
\begin{enumerate}[(i)]
\item If $x,y\in T_k^{\left(b^d\right)}$ have $\labs{x} = \labs{y}$ then $\labs{\phi(x)} = \labs{\phi(y)}$.
\item If $x\in \inte T_k^{\left(b^d\right)}$ then $\phi$ maps each of the $b^d$ distinct children of $x$ to a descendant of a distinct diagonal child of $\phi(x)$.
\end{enumerate}
The image of $\phi$ is called a replica of $T_k^{\left(b^d\right)}$. Together conditions (i) and (ii) ensure that there exists a sequence of levels
\[ \mathbf{i}_0 < \cdots < \mathbf{i}_k\]
strictly increasing in each coordinate called the \textit{signature of $\phi$} such that $\phi$ maps level $t$ of $T_k^{\left(b^d\right)}$ into level $\mathbf{i}_t$ of $\Gamma_{\mathbf{n}}^{(b)}$. The embedding and replica are called \textit{arithmetic} if there exists $\mathbf{a}\in \Z^{d}_{\geq 1}$ such that the signature is of the form
\[ \mathbf{i}_t = \mathbf{i}_0 + t \mathbf{a} \quad \text{for each }t=0,1,\ldots,k.\] 
\end{defn}
Set
\[ \alpha_{\text{crit}}(d,b):= d-1 + \log_b(b-1)\]
and write $\mathbf{1} = (1,\ldots,1)\in \Z^d$. 
The following is a refined form of Theorem~\ref{Thm: Arith replicas in general leaf generated product trees}, with the additional conclusion that the common difference of the signature may be chosen in the diagonal direction.
\begin{thm}[Refined form of Theorem~\ref{Thm: Arith replicas in general leaf generated product trees}]\label{Thm: Main general theorem}
For every $b\geq 2$, $d\geq 1$, $k\geq 1$ and $\alpha > \alpha_{\text{crit}}(d,b)$ there exists $N_0 = N_0(b,d,k,\alpha)$ such that whenever $n\geq N_0$, every
\[ A\subset \partial \Gamma_{(n,\ldots,n)}^{(b)} \qquad \text{with} \qquad \labs{A}\geq b^{\alpha n} \]
has an ancestor closure $\Gamma_A$ containing an arithmetic replica of $T_k^{\left(b^d\right)}$. Moreover, the replica can be chosen so that its signature has the form
\[ \mathbf{v} + t Q \mathbf{1} \qquad \text{for each } t=0,1,\ldots,k\]
for some $Q\geq 1$.
\end{thm}
Our proof in the binary planar case extends mutatis mutandis to the general $b$-ary $d$-dimensional case of Theorem~\ref{Thm: Main general theorem}. Let $q\geq 1$ and $L\geq 1$. In this setting our block alphabet
\[ \Omega_q: = \underbrace{(\Lambda_b^q)\times \cdots \times (\Lambda_b^q)}_{d \text{ times}}\]
has cardinality $\labs{\Omega_q} = b^{dq}$ so that the block tree $\Omega_q^{\leq L}$ is a complete $b^{dq}$-ary tree of height $L$, whose words are written in letters that are $d$-tuples of $b$-ary words of length $q$. As before we define the block map
\[ \Theta_q: \Omega_q ^{\leq L} \to \Gamma_{(qL,\ldots,qL)}^{(b)}\]
by concatenating blocks separately in each coordinate. Explicitly, if
\[ \omega = \bm{\alpha}^{(1)}\ldots\bm{\alpha}^{(t)}, \qquad \text{where each} \qquad \bm{\alpha}^{(i)} = \left(\alpha_1^{(i)},\ldots,\alpha_d^{(i)}\right) \in \Omega_q,\]
then
\[ \Theta_q(\omega) := \left(\alpha_1^{(1)}\ldots \alpha_1^{(t)}, \ldots,\alpha_d^{(1)}\ldots \alpha_d^{(t)}\right).\]
The block map $\Theta_q$ is then an order isomorphism between $\Omega_q^{\leq L}$ and the $q$-synchronous levels of $\Gamma_{(qL,\ldots,qL)}^{(b)}$, i.e. levels of the form $(qt,\ldots,qt)$ for integers $0\leq t \leq L$.

The value of $\alpha_{\text{crit}}(d,b)$ is exactly the critical growth rate above which we can find a $q$ such that the block encoding of $\Gamma_A$ must contain an arithmetic replica of a finite $S_q$-ary tree, where $S_q$ is exactly the intermediate arity needed for Proposition~\ref{Prop: General splitting lemma} to force the local $b$-ary $d$-dimensional splitting vertices.
\begin{lem}[Critical growth rate]\label{Lemma: critical growth rate}
Let $b\geq 2$ and $d\geq 1$. For each $q\geq 1$ define
\[ C_q: = b^{dq} - (b^q - (b-1)^q)^d, \qquad B_q:=b^{dq}, \qquad \text{and }S_q:= C_q+1.\]
Then for any $\alpha > \alpha_{\text{crit}}(d,b)$ there exists $q\geq 1$ such that
\begin{equation}\label{eq: q for crit dim}
\log_{B_q}(S_q-1) < \frac{\alpha}{d}.
\end{equation}
Consequently, for every $\alpha > \alpha_{\text{crit}}(d,b)$ there exists a block length $q=q(b,d,\alpha)$ such that for every $N\geq 1$ and all sufficiently large $L$, if $n=qL$ and
\[ A\subset \partial \Gamma_{(qL,\ldots,qL)}^{(b)} \qquad \text{has} \qquad \labs{A}\geq b^{\alpha n}\]
then the ancestor closure $T_{\tilde{A}}\subset \Omega_q^{\leq L}$ of the pullback $\tilde{A}:= \Theta_q^{-1}(A)$ contains an arithmetic replica of $T_N^{(S_q)}$.
\end{lem}
\begin{proof}
Let
\[ z_q:= \left(\frac{b-1}{b}\right)^q \qquad \text{so that}\qquad C_q = b^{dq} \left[1-(1-z_q)^d \right].\]
Factorising $1-(1-z_q)^d$ by geometric summation we may write
\[
C_q = b^{dq} z_q \sum_{i=0}^{d-1}(1-z_q)^i = b^{q \alpha_{\text{crit}}(d,b)} \sum_{i=0}^{d-1}(1-z_q)^i.
\]
Since $z_q \in (0,1)$ the sum lies between $1$ and $d$ so
\[ b^{q \alpha_{\text{crit}}(d,b)} \leq C_q \leq d b^{q \alpha_{\text{crit}}(d,b)}.\] 
Taking logarithms base $b$ and dividing by $q$ yields
\[ \alpha_{\text{crit}}(d,b) \leq \frac{\log_b(C_q)}{q} \leq \alpha_{\text{crit}}(d,b) + \frac{\log_b(d)}{q}\]
so
\[ \lim_{q\to \infty} \frac{\log_b(C_q)}{q} = \alpha_{\text{crit}}(d,b).\]
Since
\[ d \log_{B_q}(S_q -1) = \frac{\log_b(C_q)}{q}\]
then for any $\alpha > \alpha_{\text{crit}}(d,b)$ we can find $q=q(b,d,\alpha)$ such that \eqref{eq: q for crit dim} holds. If $A\subset \partial \Gamma_{(qL,\ldots,qL)}^{(b)}$ and $\tilde{A}\subset \partial \Omega_q^{\leq L}$ are as in the statement of the lemma, then
\[ |\tilde{A}| = \labs{A} \geq b^{\alpha qL} = B_q^{(\alpha/d)L}\]
so Theorem~\ref{Thm: Mixed arity FW} ensures that $T_{\tilde{A}}$ contains an arithmetic replica of $T_N^{(S_q)}$ for $L$ large enough.
\end{proof}
Our value of $C_q$ in Lemma~\ref{Lemma: critical growth rate} is exactly $C_{d,b}(q,\ldots,q)$ from the statement of Proposition~\ref{Prop: General splitting lemma}, so that any subset of our block alphabet $\Omega_q = \partial \Gamma_{(q,\ldots,q)}^{(b)}$ of size at least $S_q = C_q +1$ generates a $b$-ary $d$-dimensional product splitting.

The block-map properties encapsulated in Lemmata \ref{Lemma: Block branching gives local splitting} and \ref{Lem: S-ary tree in block tree gives nearly arithmetic 4-ary tree} extend without issue to the more general $b$-ary $d$-dimensional block map, and so the remainder of the proof of Theorem~\ref{Thm: Main general theorem} exactly mirrors the binary planar case. A dictionary of the parameter replacements is given below in Table~\ref{Table: Parameter replacements}.

\begin{table}[H]
    \centering
    \small
    \renewcommand{\arraystretch}{1.2}
    \begin{tabular}{@{}lcc@{}}
        \toprule
        Parameter
        & Binary planar case
        & General case \\
        \midrule
        Block alphabet $\Omega_q$
        & $\{0,1\}^q\times \{0,1\}^q$
        & $(\Lambda_b^q)^d$ \\

        Block arity $B_q$
        & $4^q$
        & $b^{dq}$ \\

        Split-free threshold $C_q$
        & $2^{q+1}-1$
        & $C_q$ \\

        Intermediate arity $S_q$
        & $2^{q+1}$
        & $C_q+1$ \\

        Arity after pruning
        & $4$
        & $b^d$ \\

        Relative splitting displacements
        & $\{0,\ldots,q-1\}^2$
        & $\{0,\ldots,q-1\}^d$ \\

        Number of colours
        & $q^2$
        & $q^d$ \\

        Critical exponent
        & $1$
        & $d-1+\log_b(b-1)$ \\
        \bottomrule
    \end{tabular}
    \caption{Parameters in the binary planar and general block
    constructions.}
    \label{Table: Parameter replacements}
\end{table}

\begin{rema}[Sharpness of $\alpha_{\text{crit}}(d,b)$]
The strict inequality $\alpha > \alpha_{\text{crit}}(d,b)$ in the statement of Theorem~\ref{Thm: Main general theorem} is necessary. Indeed, at the critical exponent, a leaf set need not even generate a single full product splitting. Fix $b\geq 2$ and $d\geq 1$. Let $W = \{0,1,\ldots,b-2\} \subset \Lambda_b$ and for each $n\geq 1$ consider the leaf set
\[ A:= W^n \times \underbrace{\Lambda_b^{n} \times \cdots \times \Lambda_b^{n}}_{d-1 \text{ times}} \subset \partial \Gamma_{(n,\ldots,n)}^{(b)}\]
of all $d$-tuples of words of length $n$ in the alphabet $\Lambda_b$ such that the word in the first component does not contain the symbol $b-1$. The set $A$ has cardinality
\[\labs{A} = (b-1)^{n} b^{(d-1)n} = b^{\alpha_{\text{crit}}(d,b)n}\]
yet
\[ \Split_d^{(b)}(A) = \emptyset.\]
Indeed, for any $\mathbf{u}=(u_1,\ldots,u_d) \in \Gamma_{A}$ with $\labs{u_i}< n$ for each $i=1,\ldots,d$, every diagonal child of $\mathbf{u}$ whose first coordinate is $u_1 (b-1)$ is not in $\Gamma_A$, since no leaf in $A$ contains the digit $(b-1)$ in its first coordinate.
\end{rema}
\section{Replicas in dense subsets of products of trees}
We define the \textit{weight} of a subset $H\subset \Gamma_{\mathbf{n}}^{(b)}$ to be
\[ w_b^d(H) := \sum_{(x_1,\ldots,x_d) \in H} b^{-(\labs{x_1} + \cdots +\labs{x_d})}.\]
Thus $d(H)=w_b^d(H)/\prod_{i=1}^d(n_i+1)$. The following is a refined form of Theorem~\ref{Th: Product density result}, again with the additional conclusion that the common difference of the signature can be chosen in the diagonal direction.
\begin{thm}[Refined form of Theorem~\ref{Th: Product density result}]\label{Thm: Refined product density result}
For all $b\geq 2$, $d,k\geq 1$, and $\delta>0$, there exists $N=N(b,d,k,\delta)$ such that for every $\mathbf{n} = (n_1,\ldots,n_d) \in \Z_{\geq 0}^{d}$ with $\min_{i} n_i \geq N$, every $H \subset \Gamma_{\mathbf{n}}^{(b)}$ with
\[ w_b^d(H) \geq \delta \prod_{i=1}^{d} (n_i +1)\]
contains an arithmetic replica of $T_k^{\left(b^d\right)}$. Moreover, the replica can be chosen so that its signature has the form
\[ \mathbf{v} + tQ \mathbf{1} \qquad \text{for each }t=0,1,\ldots,k\]
for some $Q\geq 1$. 
\end{thm}
\begin{proof}
Partition the product levels of $\Gamma_{\mathbf{n}}^{(b)}$ into maximal diagonal chains
\[ C_\mathbf{a} = \{ \mathbf{a} + t \mathbf{1} \, : \, 0\leq t <l_{\mathbf{a}}\}\]
where
\[ 0\leq \mathbf{a} \leq \mathbf{n}, \qquad \min_{i} a_i = 0, \qquad \text{ and } l_{\mathbf{a}} = 1+ \min_{i}(n_i-a_i).\]
Each product level belongs to exactly one such chain. By a union bound over the coordinate faces, the number of such chains $R$ satisfies
\[ R \leq \sum_{i=1}^d \prod_{j \neq i} (n_j+1) \leq \frac{d \prod_{i=1}^d (n_i+1)}{m}\]
where $ m: = 1+ \min_{i}n_i$. For each chain let
\[ H_{\mathbf{a}}= \{\mathbf{x} \in H \,: \, \labs{\mathbf{x}} \in C_{\mathbf{a}}\}.\]
Since these sets partition $H$, we have that $w_b^d(H) = \sum_{\mathbf{a}} w_b^d(H_{\mathbf{a}})$, and so there must be one chain $C_{\mathbf{a}}$ with length $l:= l_\mathbf{a}$ such that
\[ w_b^d(H_{\mathbf{a}}) \geq \frac{w_b^d(H)}{R} \geq \frac{\delta m\prod_{i=1}^d (n_i+1)}{d \prod_{i=1}^d(n_i+1)} = \frac{m\delta}{d}.\]
The total weight of each product level in $\Gamma_{\mathbf{n}}^{(b)}$ is $1$, and each $C_{\mathbf{a}}$ has length at most $m$, so $w_b^d(H_{\mathbf{a}}) \leq l \leq m$. Hence
\[ l\geq \frac{\delta m}{d} \qquad \text{and} \qquad w_b^d (H_{\mathbf{a}}) \geq \frac{\delta l}{d}.\]
In words, the chain is long and carries a positive proportion of its maximum possible weight.

Now write
\[ \mathcal{L}_{\mathbf{a}} = \Lambda_b^{a_1} \times \cdots \times \Lambda_b^{a_d}\]
for level $\mathbf{a}$ of $\Gamma_{\mathbf{n}}^{(b)}$. Let $\Omega = \Lambda_b^d$ be a block alphabet.\footnote{$\Omega$ is exactly $\Omega_q$ with $q=1$ from Section~\ref{Sec: general case}.} For each $\mathbf{u} = (u_1,\ldots,u_d) \in \mathcal{L}_{\mathbf{a}}$, define a map
\[\phi_{\mathbf{u}}: \Omega^{\leq l-1} \to \Gamma_{\mathbf{n}}^{(b)}\]
by concatenating blocks separately in each coordinate and then prefixing by $\mathbf{u}$. Explicitly, if
\[ \omega = \bm{\alpha}^{(1)}\ldots \bm{\alpha}^{(t)}, \qquad \text{where each} \qquad \bm{\alpha}^{(i)} = \left(\alpha_1^{(i)},\ldots,\alpha_d^{(i)}\right) \in \Omega\] 
then
\[ \phi_{\mathbf{u}}(\omega) = \left(u_1 \alpha_1^{(1)}\ldots \alpha_1^{(t)},\ldots,u_d \alpha_d^{(1)}\ldots \alpha_d^{(t)}\right).\]
The map $\phi_{\mathbf{u}}$ sends level $t$ of $\Omega^{\leq l-1}$ to level $\mathbf{a} + t \mathbf{1}$, and the $b^d$ children of any non-leaf $x\in \Omega^{\leq l-1}$ are mapped to the diagonal children of $\phi_{\mathbf{u}}(x)$.

Let
\[ K_{\mathbf{u}}:= \phi_{\mathbf{u}}^{-1}(H). \]
As $\mathbf{u}$ varies over $\mathcal{L}_{\mathbf{a}}$, the images of $\phi_{\mathbf{u}}$ partition all product vertices on $C_{\mathbf{a}}$ and so
\begin{equation}\label{eq: computing weights through phi_u}
w_b^d(H_{\mathbf{a}}) = \sum_{\mathbf{u}\in \mathcal{L}_{\mathbf{a}}} w_b^d(\phi_{\mathbf{u}}(K_{\mathbf{u}})) = b^{-(a_1 + \cdots + a_d)} \sum_{\mathbf{u} \in \mathcal{L}_{\mathbf{a}}} w_{b^d}(K_{\mathbf{u}})
\end{equation}
where the second equality follows since any $x \in \Omega^{\leq l-1}$ with $\labs{x} = t$ satisfies
\[ w_b^d(\{\phi_{\mathbf{u}}(x)\}) = b^{-\left(a_1+\cdots+a_d + dt\right)} = b^{-(a_1+\cdots+a_d)} w_{b^d}(\{x\}).\]
Since $\labs{\mathcal{L}_{\mathbf{a}}}=b^{a_1+\cdots+a_d}$, the right-hand side of equation~\eqref{eq: computing weights through phi_u} is exactly the average of the weights $w_{b^d}(K_{\mathbf{u}})$. Hence some $\mathbf{u}\in \mathcal{L}_{\mathbf{a}}$ must satisfy
\[ w_{b^d}(K_{\mathbf{u}}) \geq w_b^d(H_{\mathbf{a}}) \geq \frac{\delta l}{d}.\]
Since $l \geq (\delta/d)(1+\min_i n_i)$, Theorem~\ref{Thm: Equal arity density} ensures that $K_{\mathbf{u}}$ contains an arithmetic replica of $T_k^{\left(b^d\right)}$ provided that $\min_i n_i$ is large enough. The image of this replica under $\phi_{\mathbf{u}}$ is then an arithmetic replica of $T_k^{\left(b^d\right)}$ inside $H$ whose signature has common difference proportional to $\mathbf{1}$, as required.
\end{proof}

\end{document}